\RequirePackage{fix-cm}
\documentclass[smallextended]{svjour3}       
\smartqed  
\usepackage{amssymb,amsmath,amsfonts,mathrsfs}
\usepackage[utf8]{inputenc}
\usepackage[english]{babel}
\usepackage{graphicx}
\usepackage{color}
\DeclareMathOperator{\rank}{rank}
\DeclareMathOperator{\cone}{cone.hull}
\DeclareMathOperator{\conv}{conv.hull}

\DeclareMathOperator{\width}{width}
\DeclareMathOperator{\vertes}{vert}

\DeclareMathOperator{\size}{size}

\DeclareMathOperator{\poly}{poly}
\DeclareMathOperator{\mult}{mult}

\DeclareMathOperator{\inter}{int}

\DeclareMathOperator{\parl}{par}
\DeclareMathOperator{\vol}{Vol}

\begin{document}

\title{FPT-algorithms for some problems related to integer programming}


\author{D.~V.~Gribanov \and D.~S.~Malyshev \and P.~M.~Pardalos \and S.~I.~Veselov}


\institute{D.~V.~Gribanov \at Lobachevsky State University of Nizhny Novgorod, 23 Gagarina Avenue, Nizhny Novgorod, 603950, Russian Federation\\
National Research University Higher School of Economics, 25/12 Bolshaja Pecherskaja
Ulitsa, Nizhny Novgorod, 603155, Russian Federation\\
\email{dimitry.gribanov@gmail.com}
\and D.~S.~Malyshev \at National Research University Higher School of Economics, 25/12 Bolshaja Pecherskaja
Ulitsa, Nizhny Novgorod, 603155, Russian Federation\\
\email{dsmalyshev@rambler.ru}
\and P.~M.~Pardalos \at University of Florida, 401 Weil Hall, P.O. Box 116595,
Gainesville, FL 326116595, USA
\at National Research University Higher School of Economics, 25/12 Bolshaja Pecherskaja
Ulitsa, Nizhny Novgorod, 603155, Russian Federation\\
\email{p.m.pardalos@gmail.com}
\and S.~I.~Veselov \at Lobachevsky State University of Nizhny Novgorod, 23 Gagarina Avenue, Nizhny Novgorod, 603950, Russian Federation\\
\email{sergey.veselov@itmm.unn.ru}
}

\date{Received: date / Accepted: date}

\maketitle

\begin{abstract}
In this paper, we present fixed-parameter tractable algorithms for special cases of the shortest lattice vector, integer linear programming, and simplex width computation problems, when matrices included in the problems' formulations are near square. The parameter is the maximum absolute value of the rank minors in the corresponding matrices. Additionally, we present fixed-parameter tractable algorithms with respect to the same parameter for the problems, when the matrices have no singular rank submatrices.

\end{abstract}

\section{Introduction}
Let $A \in \mathbb{Z}^{d \times n}$ be an integer matrix. We denote by $A_{ij}$ the $ij$-th element of the matrix, by $A_{i*}$ its $i$-th row, and by $A_{*j}$ its $j$-th column. The set of integer values starting from $i$ and ending in $j$ is denoted by $i:j=\left\{i, i+1, \ldots, j\right\}$. Additionally, for subsets $I \subseteq \{1,\dots,d\}$ and $J \subseteq \{1,\dots,n\}$, $A_{I\,J}$ denotes the submatrix of $A$ that was generated by all rows with numbers in $I$ and all columns with numbers in $J$. When $I$ or $J$ are replaced by $*$, that implies that all rows or columns (respectively) are selected. By $0_{m\times n}$, we mean
a $m\times n$-matrix with the zeroes entries only, $0$ also means the zero vector of the corresponding dimension. For example, $A_{I*}$ is the submatrix consisting of all rows in $I$ and all columns. Let $||A||_{\max}$ denote the maximum absolute value of any element in $A$. Let $\Delta_k(A)$ denote the greatest absolute value of determinants of all $k \times k$ submatrices of $A$, respectively. Additionally, let $\Delta(A) = \Delta_{\rank(A)}(A)$.

\begin{definition}
For a vector $b\in\mathbb{Z}^{n}$, by $P(A,b)$ we denote a polyhedron $\{ x \in \mathbb{R}^{n} : A x \leq b\}$. The set of all vertices of a polyhedron $P$ is denoted by $\vertes(P)$.
\end{definition}

\begin{definition}
For a matrix $B \in \mathbb{R}^{d \times n}$, $\cone(B) = \{B t : t \in \mathbb{R}_+^{n} \}$ is the \emph{cone spanned by columns of} $B$, $\conv(B) = \{B t : t \in \mathbb{R}_+^{n},\, \sum_{i=1}^{n} t_i = 1  \}$ is the \emph{convex hull spanned by columns of} $B$, $\parl(B) = \{x \in \mathbb{R}^d : x = B t,\, t \in [0,1)^n \}$ is the \emph{parallelepiped spanned by columns of} $B$, and $\Lambda(B) = \{B t : t \in \mathbb{Z}^{n} \}$ is the \emph{lattice spanned by columns of} $B$.
\end{definition}

We refer to \cite{CAS71,GRUB87,SIEG89} for mathematical introductions to lattices.


\begin{definition}
The \emph{width of a convex body} $P$ is defined as
\[
\width(P)=\min\limits_{c \in \mathbb{Z}^n\setminus\{0\}} (\max\limits_{x \in P} c^\top x - \min\limits_{x \in P} c^\top x).
\]
A vector $c$ minimizing the difference $\max\limits_{x \in P} c^\top x - \min\limits_{x \in P} c^\top x$ on $\mathbb{Z}^n\setminus\{0\}$ is called the \emph{flat direction of} $P$.
\end{definition}

\begin{definition}
Following \cite{SCHR98}, we define the sizes of an integer number $x$, a rational number $r = \frac{p}{q}$, a rational vector $v \in \mathbb{Q}^n$, and a rational matrix $A \in \mathbb{Q}^{d \times n}$ in the following way:
\begin{align*}
& \size(x) = 1 + \lceil \log_2 (x+1) \rceil,\\
& \size(r) = 1 + \lceil \log_2 (p+1) \rceil + \lceil \log_2 (q+1) \rceil,\\
& \size(v) = n + \sum_{i=1}^n \size(v_i),\\
& \size(A) = dn + \sum_{i=1}^d \sum_{j=1}^n \size(A_{i\,j}).
\end{align*}
\end{definition}

\begin{definition}
An algorithm parameterized by a parameter $k$ is called \emph{fixed-parameter tractable} (FPT-\emph{algorithm}) if its complexity can be estimated by a function from the class $f(k)\, n^{O(1)}$, where $n$ is the input size and $f(k)$ is a computable function that depends on $k$ only.
A computational problem parameterized by a parameter $k$ is called \emph{fixed-parameter tractable} (FPT-\emph{problem}) if it can be solved by a FPT-algorithm. For more information about the parameterized complexity theory, see \cite{PARAM15,PARAM99}.
\end{definition}

{\bf The shortest lattice vector problem}

The \emph{Shortest Lattice Vector Problem} (the SLVP) consists in finding $x \in \mathbb{Z}^n \setminus \{0\}$ minimizing $||H x||$, where $H \in \mathbb{Q}^{d \times n}$ is given as an input. The SLVP is known to be NP-hard with respect to randomized reductions, cf. \cite{AJTAI96}. The first polynomial-time approximation algorithm for the SLVP was proposed by A.~Lenstra, H.~Lenstra~Jr., and L.~Lov\'asz in \cite{LLL82}.  Shortly afterwards, U.~Fincke and M.~Pohst in \cite{FP83,FP85}, R.~Kannan in \cite{KANN83,KANN87} described the first exact SLVP solvers. Kannan's solver has a computational complexity of  $2^{O(n\,\log n)} \poly(\size(H))$, where $\poly(\cdot)$ means some polynomial on its argument. The first SLVP solvers that achieve the complexity $2^{O(n)} \poly(\size(H))$ were proposed by M.~Ajtai, R.~Kumar, D.~Sivakumar \cite{AJKSK01,AJKSK02}, D.~Micciancio and P.~Voulgaris \cite{MICCVOUL10}. The previously discussed SLVP solvers are used for the Euclidean norm. Recent results about SLVP-solvers for more general norms are presented in \cite{BLNAEW09,DAD11,EIS11}. The paper of G.~Hanrot, X.~Pujol, D.~Stehl\'e \cite{SVPSUR11} is a good survey about SLVP-solvers.

Recently, a novel polynomial-time approximation SLVP-solver was proposed by J.~Cheon and L.~Changmin in \cite{CHLEE15}. The algorithm is parameterized by the lattice determinant, its time-complexity and the approximation factor are the best to date for lattices with a sufficiently small determinant.

In our work, we consider only integer lattices, whose generating matrices are near square. The first aim of this paper is to present an exact FPT-algorithm for the SLVP parameterized by the lattice determinant (see Section 3). Additionally, we develop a FPT-algorithm for lattices, whose generating matrices have no singular rank submatrices. The proposed algorithms work for the $l_p$ norm for any finite $p \geq 1$ and also for the $l_\infty$ norm.

{\bf The integer linear programming problem}

The \emph{Integer Linear Programming Problem} (the ILPP) can be formulated as $\min\{ c^\top x : x \in P(H,b) \cap \mathbb{Z}^n\}$ for integer vectors $c,b$ and an integer matrix $H$.

There are several polynomial-time algorithms for solving linear programs. We mention Khachiyan's algorithm \cite{KHA80}, Karmarkar's algorithm \cite{KAR84}, and Nesterov's algorithm \cite{NN94,PAR91}. Unfortunately, it is well known that the ILPP is NP-hard, in the general case. Therefore, it would be interesting to reveal polynomially solvable cases of the ILPP. An example of this type is the ILPP with a fixed number of variables, for which a polynomial-time algorithm is given by H.~Lenstra in \cite{LEN83}. Another examples can be obtained, when we add some restrictions to the structure of constraints matrices. A square integer matrix is called \emph{unimodular} if its determinant equals $+1$ or $-1$. An integer matrix is called \emph{totally unimodular} if all its minors are $+1$ or $-1$ or $0$. It is well known that all optimal solutions of any linear program with a totally unimodular constraints matrix are integer. Hence, for any linear program and the corresponding integer linear program with a totally unimodular constraints matrix, the sets of their optimal solutions coincide. Therefore, any polynomial-time linear optimization algorithm (like the ones in \cite{KAR84,KHA80,NN94,PAR91}) is also an efficient algorithm for the ILPP.

The next natural step is to consider the \emph{totally bimodular} case, i.e. the ILPP having constraints matrices with the absolute values of all rank minors in the set $\{0, 1, 2\}$. The first paper that discovers fundamental properties of the bimodular ILPP is the paper of S.~I.~Veselov and A.~Y.~Chirkov \cite{VESCH09}. Very recently, using results of \cite{VESCH09}, a strong polynomial-time solvability of the bimodular ILPP was proved by S.~Artmann, R.~Weismantel, R.~Zenklusen in \cite{AW17}. A matrix will be called \emph{totally $\Delta$-modular} if all its rank minors are at most $\Delta$ in the absolute value.

More generally, it would be interesting to investigate the computational complexity of the problems with bounded minors constraints matrices. The maximum absolute value of rank minors of an integer matrix can be interpreted as a proximity measure to the class of totally unimodular matrices. Let the symbol ILPP$_{\Delta}$ denote the ILPP with constraints matrix, each rank minor of which has the absolute value at most $\Delta$. In \cite{SHEV96},  a conjecture is presented that for each fixed natural number $\Delta$ the ILPP$_{\Delta}$ can be solved in polynomial-time. There are variants of this conjecture, where the augmented matrices $\dbinom{c^\top}{A}$ and $(A \; b)$ are considered \cite{AZ11,SHEV96}.

Unfortunately, not much is known about the computational complexity of the ILPP$_{\Delta}$. For example, the complexity status of the ILPP$_{3}$ is unknown. A step towards deriving the its complexity was done by Artmann et al. in \cite{AE16}. Namely, it has been shown that if the constraints matrix, additionally, has no singular rank submatrices, then the ILPP$_{\Delta}$ can be solved in polynomial-time. Some results about polynomial-time solvability of the boolean ILPP$_{\Delta}$ were obtained in \cite{AZ11,BOCK14,GRIBM17}. F.~Eisenbrand and S.~Vempala \cite{EIS16} presented a randomized simplex-type linear programming algorithm, whose expected running time is strongly polynomial if all minors of the constraints matrix are bounded by a fixed constant.

In \cite{GRIB13,GRIBV16}, it has been shown that any lattice-free polyhedron of the ILPP$_{\Delta}$ has a relatively small width, i.e., the width is bounded by a function that is linear on the dimension and exponential on $\Delta$. Interestingly, due to \cite{GRIBV16}, the width of any empty lattice simplex can be estimated by $\Delta$, for this case. In \cite{GRIBC16}, it has been shown that the width of any simplex induced by a system, having the absolute values
of minors bounded by a fixed constant, can be computed by a polynomial-time algorithm. As it was mentioned in \cite{AW17}, due to E.~Tardos' results \cite{TAR86}, linear programs with constraints matrices, whose all minors are bounded by a fixed constant, can be solved in strongly polynomial time. N.~Bonifas et al. \cite{BONY14} showed that any polyhedron defined by a totally $\Delta$-modular matrix has a diameter bounded by a polynomial on $\Delta$ and the number of variables.

The second aim of our paper is to improve results of \cite{AE16}. Namely, in Section 4, we will present a FPT-algorithm for the ILPP$_{\Delta}$, when the constraints matrix is close to a square matrix, i.e. it has a fixed number of additional rows. This fact gives us a FPT-algorithm for the case, when the problem's constraints matrix has no singular rank submatrices. Indeed, such matrices can have only one additional row if the dimension is sufficiently large, due to \cite{AE16}. In this paper, we present an algorithm with a better complexity bound. Additionally, we improve some inequalities established in \cite{AE16}.\\
{\bf Computing the simplex lattice width}

A.~Seb\"{o} shown \cite{SEB99} that the problem of computing the rational simplices width is NP-hard. A.~Y.~Chirkov and D.~V.~Gribanov \cite{GRIBC16} shown that the problem can be solved by a polynomial-time algorithm in the case, when the simplex is defined by a bounded minors constraints matrix. The final aim of this paper is to present a FPT-algorithm for the simplex width computation problem (see Section 5).

\section{Some auxiliary results}

Let $H$ be a $d\times n$ matrix of rank $n$ that has already been reduced to the Hermite normal form (the HNF) \cite{SCHR98,STORH96,ZHEN05}. Let us assume, without loss of generality, that the matrix $H_B = H_{1:n\,*}$ is non-singular, and let $H_N$ be the $m \times n$ matrix generated by the remaining columns of $H$. In other words, $H = \dbinom{H_B}{H_N}$ and $d = n + m$.

Using additional permutations of rows and columns, we can transform $H$, such that the matrix $H_B$ has the following form:

\begin{equation} \label{HNF}
H_B = \begin{pmatrix}
1            & 0                   & \dots         & 0           & 0               & 0      & \dots & 0\\
0            & 1                   & \dots         & 0           & 0              &0       & \dots & 0\\
\hdotsfor{8} \\
0            &        0            & \dots         & 1           & 0           & 0          & \dots & 0\\
H_{s+1\,1}  &   H_{s+2\,2}        & \dots        & H_{s+1\,s}  & H_{s+1\,s+1} & 0           & \dots & 0\\
\hdotsfor{8} \\
H_{n\,1}  &   H_{n\,2}        & \hdotsfor{5}  & H_{n\,n}\\
\end{pmatrix},
\end{equation}

\noindent where $s$ is the number of 1's on the diagonal. Hence, $H_{i\,i} \geq 2$, for $i \in (s+1) : n$. Let, additionally, $k = n - s$ be the number of diagonal elements that are not equal to $1$, $\Delta = \Delta(A)$ and $\delta = |\det(H_B)|$.

The following properties are known for the HNF:
\begin{itemize}
\item[1)] $0 \leq H_{i\,j} < H_{i\,i}$, for any $i \in 1:n$ and $j \in 1:(i-1)$,

\item[2)] $\Delta \geq \delta = \prod_{i=s+1}^n H_{i\,i}$, and, hence, $k \leq \log_2 \Delta$,

\item[3)] since $H_{i\,i} \geq 2$, for $i \in (s+1) : n$, we have \[\sum\limits_{i=s+1}^n H_{i\,i} \leq \frac{\delta}{2^{k-1}} + 2(k-1) \leq \delta.\]
\end{itemize}

%

In \cite{AE16}, it was shown that $||H_N||_{\max} \leq a_q$, where $q = \lceil \log_2 \Delta \rceil$, and the sequence $\{a_i\}$ is defined, for $i \in 0:q$, as follows:
\[
a_0 = \Delta,\quad a_i = \Delta + \sum_{j=0}^{i-1} a_j \Delta^{\log_2 \Delta} (\log_2 \Delta)^{(\log_2 \Delta /2)}.
\]

It is easy to see that $a_q = \Delta (\Delta^{\log_2 \Delta} (\log_2 \Delta)^{(\log_2 \Delta /2)}+1)^{\lceil \log_2 \Delta \rceil}$.

We will show that the estimate on $||H_N||_{\max}$ can be significantly improved.

\begin{lemma}\label{HNFElem}
\[||H_N||_{\max} \leq \frac{\Delta}{\delta} ( \frac{\delta}{2^{k-1}} + k -1) \leq \Delta.\]

Hence, $||H||_{\max} \leq \Delta$.
\end{lemma}
\begin{proof} Let $h = H_{i\,*}$, for $i \in (n+1) : d$, and $h = t^\top H_B$, for some $t \in \mathbb{R}^n$. Let $H(j)$ be the matrix obtained from $H_B$ by replacing $j$-th row with row $h$. For any $j \in 1 : n$, we have $\det(H(j)) = t_j \det(H_B)$, hence, $|t_j| \leq \frac{\Delta}{\delta}$. Using the property 3) of the HNF, we have $$|H_{i\,j}| = |h_j| \leq \sum_{l=1}^n |t_l H_{l\,j}| < \frac{\Delta}{\delta} (1 + \sum_{ l = s+1}^n H_{l\,l} - k) \leq \frac{\Delta}{\delta} ( \frac{\delta}{2^{k-1}} + k -1).$$
\end{proof}

We also need the following technical lemma: \begin{lemma}\label{SubRankDet}
Let $H$ be an $(n+1) \times n$ integer matrix of rank $n$ that has already been reduced to the HNF, and it has the form \eqref{HNF}. Then $\Delta_{n-1}(H) \leq \frac{\Delta^2}{2} (1 + \log_2 \Delta)$.
\end{lemma}
\begin{proof}
Let the matrix $A$ be obtained from $H$ by deleting any two rows and any column. It is easy to see that $A$ is a lower triangular matrix with at most one additional diagonal. We can expand the determinant of $A$ by the first row, using the Laplace theorem. Then, $|\det(A)| \leq 2^k |d_1 d_2 \dots d_{k-1} c|$, where $k$ is the number of non-zero diagonal elements in $H_B$, $\{d_1,d_2,\dots,d_{k}\}$ is the sequence of diagonal elements resp., and $c = d_k$ or $c$ is some element of the last row of $H$. Since $|d_k| \geq 2$, we have $|d_1 d_2 \dots d_{k-1}| \leq \delta/2$. Lemma \ref{HNFElem} provides us with an estimate on $|c|$.
Finally, we have
\[
|\det(A)| \leq 2^{k-1} \Delta ( \frac{\delta}{2^{k-1}} + k -1) \leq \frac{\delta \Delta}{2} (1 + \log_2 \delta).
\]
\end{proof}

Let the matrix $H$ have the additional property, such that $H$ has no singular $n \times n$ submatrices. One result of \cite{AE16}  states that if $n \geq f(\Delta)$, then the matrix $H$ has at most $n+1$ rows, where $f(\Delta)$ is a function that depends on $\Delta$ only. The paper \cite{AE16} contains a super-polynomial estimate on the value of $f(\Delta)$. Here, we will show the existence of a polynomial estimate.

\begin{lemma}\label{NRowsHNF}
If $n > \Delta  (2 \Delta +1)^2 + \log_2 \Delta$, then $H$ has at most $n+1$ rows.
\end{lemma}
\begin{proof} Our proof of the theorem has the same structure and ideas as in \cite{AE16}. We employ Lemma \ref{HNFElem} with a slight modification.

Let the matrix $H$ be defined as illustrated in \eqref{HNF}. Recall that $H$ has no singular $n \times n$ submatrices. For the purpose of deriving a contradiction, assume that $n > \Delta  (2 \Delta +1)^2 + \log_2 \Delta$ and $H$ has exactly $n+2$ rows. Let again, as in \cite{AE16}, $\bar H$ be the submatrix of $H$ without rows indexed by numbers $i$ and $j$, where $i,j \leq s$ and $i > j$. Observe, that
\[
|\det \bar H| = |\det \underbrace{\begin{pmatrix}
H_{s+1\,i} & H_{s+1\,j}& H_{s+1\,s+1} &                   &    \\
\vdots   &\vdots    &              &   \ddots       &    \\
H_{n\,i}& H_{s\,j}&       \hdotsfor{2}          & H_{n\,n} \\
H_{n+1\,i}& H_{n+1\,j} &      \hdotsfor{2}          & H_{n+1\,n} \\
H_{n+ 2\,i}& H_{n+2\,j} &      \hdotsfor{2}          & H_{n+2\,n} \\
\end{pmatrix}}_{:={\bar H}^{ij}}|.
\]

The matrix ${\bar H}^{ij}$ is a non-singular $(k+2)\times(k+2)$-matrix. This implies that the first two columns of ${\bar H}^{ij}$ must be different, for any $i$ and $j$. By Lemma \ref{HNFElem} and the structure of the HNF, there are at most $\Delta \cdot  (2 \Delta +1)^2$ possibilities to choose the first column of ${\bar H}^{ij}$. Consequently, since $n > \Delta (2 \Delta +1)^2 + \log_2 \Delta$, then $s > \Delta (2 \Delta +1)^2$, and there must exist two indices $i \not= j$, such that $\det {\bar H}^{ij} = 0$. This is a contradiction.
\end{proof}

\section{A FPT-algorithm for the shortest lattice vector problem}

Let $H \in \mathbb{Z}^{d \times n}$. The SLVP related to the $l_p$ norm can be formulated as follows:

\begin{equation}\label{ISVP}
\min\limits_{x \in \Lambda(H) \setminus \{0\} } ||x||_p,
\end{equation} or equivalently
\begin{align*}
&||x||_p \to \min\\
&\begin{cases}
x = H t \\
t \in \mathbb{Z}^n \setminus \{0\}.
\end{cases}
\end{align*}

Since there is a polynomial-time algorithm to compute the HNF, we can assume that $H$ has already been reduced to the form \eqref{HNF}.

\begin{theorem}\label{SimpleSVP}
If $n > \Delta (2 \Delta + 1)^m + \log_{2} \Delta$, then there exists a polynomial-time algorithm to solve the problem \eqref{ISVP} with a bit-complexity of $O(n \log n \cdot \log \Delta (m + \log \Delta))$.
\end{theorem}

\begin{proof}
Since $n = s + k$ and $k \leq \log_2 \Delta$, we have $s > \Delta (2 \Delta + 1)^m$. Consider the matrix $\bar H = H_{*\,1:s}$ that consists of the first $s$ columns of the matrix $H$. By Lemma \ref{HNFElem}, there are strictly less than $\Delta \cdot (2 \Delta + 1)^m$ possibilities to generate a column of $\bar H$, so if $s > \Delta (2 \Delta + 1)^m$, then $\bar H$ has two equivalent columns. Hence, the lattice $\Lambda(H)$ contains the vector $v$, such that $||v||_p = \sqrt[p]{2}$ (and $||v||_\infty = 1$). We can find equivalent rows of $\bar H$, using any sorting algorithm with the number of lexicographical comparisons $O(n \log n)$, where a bit-complexity of the two vectors lexicographical comparison operation is of $O(\log \Delta (m + \log \Delta))$. Finally, it is easy to see that the lattice $\Lambda(H)$ contains a vector of the $l_p$ norm $1$ (for $p \not= \infty$) if and only if the matrix $\bar H$ contains the zero column.
\end{proof}

In the case, when $m = 0$ and $H$ is a square non-singular matrix, we have the following trivial corollary:
\begin{corollary}
If $n \geq \Delta +\log_2{\Delta}$, then there exists a polynomial-time algorithm to solve problem \eqref{ISVP} with a bit-complexity of $O(n\log{n}\cdot \log^2{\Delta})$.
\end{corollary}

Let $x^*$ be an optimal vector of the problem \eqref{ISVP}. The classical Minkowski's theorem in geometry of numbers states that:
\[
||x^*||_p \leq 2 \left(\frac{\det \Lambda(H)}{\vol(B_p)}\right)^{1/n},
\]
where $B_p$ is the unit sphere for the $l_p$ norm.

Using the inequalities $\det \Lambda(H) = \sqrt{\det H^\top H} \leq \Delta \sqrt{\dbinom{d}{n}} \leq \Delta \left(\cfrac{e d}{n}\right)^{n/2}$, we can conclude  that $||x^*||_p \leq 2 \sqrt{\cfrac{e d}{n}}  \sqrt[n]{\cfrac{\Delta}{\vol(B_p)}}$.

On the other hand, by Lemma \ref{HNFElem}, the last column of $H$ has the norm equals $\Delta \sqrt[p]{m+1}$. Let
\begin{equation}\label{MConst}
M = \min \Bigl\{\Delta \sqrt[p]{m+1},\, 2 \sqrt{\frac{e d}{n}}  \sqrt[n]{\frac{\Delta}{\vol(B_p)}}\Bigr\}
\end{equation} be the minimum value between these two estimates.

\begin{theorem}
There is an algorithm with a complexity of $$O( (\log \Delta + m) \cdot n^{m+1} \cdot \Delta^{m+1} \cdot M^{m+1} \cdot \mult(\log \Delta + \log n + \log M) )$$ to solve the problem \eqref{ISVP}. Since $M \leq \Delta \sqrt[p]{m+1}$ (cf. \eqref{MConst}), the problem \eqref{ISVP} parameterized by $\Delta$ is included in the FPT-complexity class, for any fixed $m$.
\end{theorem}
\begin{proof}

After splitting the variables $x$ into two groups $x_B$ and $x_N$ with relation to $H_B$ and $H_N$, the problem \eqref{ISVP} becomes:
\begin{align*}
& ||x||_p^p \to \min\\
&\begin{cases}
x_B - H_B t = 0\\
x_N - H_N t = 0\\
x_B \in \mathbb{Z}^n,\, x_N \in \mathbb{Z}^m\\
t \in \mathbb{Z}^n \setminus \{0\}.
\end{cases}
\end{align*}

Using the formula $t = H_B^{-1} x_B$, we can eliminate the variables $t$ from the restriction $x_N - H_N t = 0$. The restriction can be additionally multiplied by $\delta$ to become integer, where $H_B^* = \delta H_B^{-1}$ is the adjoint matrix for $B$.
\begin{align*}
& ||x||_p^p \to \min\\
&\begin{cases}
x_B - H_B t = 0\\
\delta x_N - H_N H_B^{*} x_B  = 0\\
x_B \in \mathbb{Z}^n,\, x_N \in \mathbb{Z}^m\\
t \in \mathbb{Z}^n \setminus \{0\}.
\end{cases}
\end{align*}

Finally, we transform the matrix $H_B$ into the Smith normal form (the SNF) \cite{SCHR98,STORS96,ZHEN05}, such that $H_B = P^{-1} S Q^{-1}$, where $P^{-1}$, $Q^{-1}$ are unimodular matrices and $S$ is the SNF of $H_B$. After applying the transformation $t \to Q t$, the initial problem becomes equivalent to the following problem:
\begin{align*}
& ||x||_p^p \to \min\\
&\begin{cases}
G x_B \equiv 0 (\text{mod }S)\\
R x_B  = \delta x_N\\
x_B \in \mathbb{Z}^n \setminus \{0\},\, x_N \in \mathbb{Z}^m\\
||x||_\infty \leq M,
\end{cases}
\end{align*}

where $G = P\text{ mod }S$, $R = H_N H_B^{*}$. The inequality $||x||_{\infty} \leq M$ is an additional tool to localize an optimal integer solution. We also have that $||R||_{\max} = ||H_N H_B^{*}||_{\max} \leq \Delta$.

Actually, the considered problem is the classical Gomory's group minimization problem \cite{GOM65} (cf. \cite{HU70}) with additional linear constraints. As in \cite{GOM65}, it can be solved using the dynamic programming approach.

To this end, let us define the subproblems $Prob(l,\gamma,\eta)$:
\begin{align*}
& ||x||_p^p \to \min\\
&\begin{cases}
G_{*\,1:l} x \equiv \gamma (\text{mod }S)\\
R_{*\,1:l} x  = \eta\\
x \in \mathbb{Z}^l \setminus \{0\},\\
\end{cases}
\end{align*}
where $l \in 1:n$, $\gamma \in \mathbb{Z}^n\text{ mod }S$, $\eta \in \mathbb{Z}^m$, and $||\eta||_{\infty} \leq n M \Delta$.

Let $\sigma(l,\gamma,\eta)$ be the objective function optimal value of $Prob(l,\gamma,\eta)$. When the problem $Prob(l,\gamma,\eta)$ is unfeasible, we put $\sigma(l,\gamma,\eta) = +\infty$. In the beginning, we put $\sigma(l,\gamma,\eta)=+\infty$, for all values $l$, $\gamma \not= 0$, $\eta \not= 0$ and $\sigma(l, 0, 0) =0$. Trivially, the optimum of \eqref{ISVP} is
\[
\min\limits_{\eta : ||\eta||_{\infty} \leq M} \{\sigma(n,0,\delta \eta) + ||\eta||^p_p\}.
\]

The following formula gives the relation between $\sigma(l,\cdot,\cdot)$ and $\sigma(l-1,\cdot,\cdot)$:
\[
\sigma(l,\gamma,\eta) = \min \{f(z) : |z| \leq  M\},
\] where
\[
f(z) = \begin{cases}
\sigma(l-1,\gamma,\eta),\text{ for } z = 0\\
|z|^p + [z R_{*\,l} \not= \eta] \cdot \sigma(l-1,\gamma - z G_{*\,l},\eta - z R_{*\,l}),
\end{cases}
\] where the symbol $[z R_{*\,l} \not= \eta]$ equals $1$ if and only if the condition $z R_{*\,l} \not= \eta$ is true. The value of $\sigma(1,\gamma,\eta)$ can be computed using the following formula:
\[
\sigma(1,\gamma,\eta) = \min \{|z|^p : z G_{*\,1} \equiv \gamma \,(\text{mod } S),\, z R_{*\,1} = \eta,\, 0 < |z| \leq M \}.
\]

Both the computational complexity of computing $\sigma(1,\gamma,\eta)$ and the reduction complexity of $\sigma(l,\gamma,\eta)$ to $\sigma(l-1,\cdot,\cdot)$, for all $\gamma$ and $\eta$, can be roughly estimated as:
\[
O( (\log \Delta + m) \cdot \Delta M \cdot (n M \Delta)^m \cdot \mult(\log \Delta + \log n + \log M) ).
\]
The final complexity result can be obtained by multiplying the last formula by $n$.
\end{proof}

Let us consider the special case, when all $n \times n$ submatrices of $H$ are non-singular. In this case, by Lemma \ref{NRowsHNF}, for $n > \Delta(2 \Delta + 1)^2 + \log_2 \Delta$, the matrix $H$ can have at most $n+1$ rows ($m \leq 1$), and we have the following corollary.

\begin{corollary}
Let $H$ be the matrix defined as illustrated in \eqref{HNF}. Let also $H$ have no singular $n \times n$ submatrices. If $n > \Delta (2 \Delta + 1)^2 + \log_2 \Delta$, then there is an algorithm with a complexity of $O(n \log n \cdot \log^2 \Delta)$ that solves the problem \eqref{ISVP}.
\end{corollary}
\begin{proof}
We have $n > \Delta (2 \Delta + 1)^2 + \log_2 \Delta > \Delta (2 \Delta +1)^m + \log_2 \Delta$. The last inequality meets the conditions of Theorem \ref{SimpleSVP}, and the corollary follows.
\end{proof}

\begin{note}
Due to the objective function separability, it is easy to see that the same approach is applicable for the Closest Lattice Vector problem (cf. \cite{SVPSUR11}), that can be formulated as follows:
\begin{equation*}
\min\limits_{x \in \Lambda(H) } ||x-r||_p,
\end{equation*} where $r \in \mathbb{Q}^n$. The resulting algorithm has the same complexity on $n$ and $\Delta$, and it is polynomial-time on $\size(H)$ and $\size(r)$.
\end{note}

\section{The integer linear programming problem}

Let $H \in \mathbb{Z}^{d \times n}$, $c \in \mathbb{Z}^n$, $b \in \mathbb{Z}^d$, $rank(H) = n$. Let us consider the ILPP:
\begin{equation}\label{IPP}
\max\{c^\top x : x \in P(H,b) \cap \mathbb{Z}^n \}.
\end{equation} Since there is a polynomial-time algorithm to compute the HNF, we can assume that $H$ has already been reduced to the form \eqref{HNF}.

\begin{theorem}\label{IPPT}
The problem \eqref{IPP} can be solved by an algorithm with a complexity of
\[
O( (\log \Delta + m) \cdot n^{2(m+1)} \cdot \Delta^{2(m+1)} \cdot \mult(\size(c) + \log \Delta) ).
\]
\end{theorem}
\begin{proof}

Let $v$ be an optimal solution of the linear relaxation of the problem \eqref{IPP}. We can suppose without loss of generality that $H_B v = b_{1:n}$. As in \cite{AE16}, after an introduction of the slack variables $y \in \mathbb{Z}^n_+$, the problem \eqref{IPP} becomes:
\begin{align*}
& c^\top x \to \max\\
&\begin{cases}
H_B x + y = b_{1:n}\\
H_N x \leq b_{(n+1) : m}\\
x \in \mathbb{Z}^n,\, y \in \mathbb{Z}^n_+.
\end{cases}\\
\end{align*}

Due to the classical result of W.~Cook, A.~Gerards, A.~Schrijver, and E.~Tardos \cite{COGST86,SCHR98}, we have that
\begin{equation}\label{TardoshT}
||y||_\infty \leq n \Delta.
\end{equation}


Now, using the formula $x = H_B^{-1} (b_{1:n} - y)$, we can eliminate the $x$ variables from the last constraint and from the objective function:
\begin{align*}
& c^\top H_B^{-1} b_{1:n} - c^\top H_B^{-1} y \to \min\\
&\begin{cases}
H_B x + y = b_{1:n}\\
-H_N H_B^{*} y \leq \delta b_{(n+1) : m} - H_N H_B^{*} b_{1:n}\\
x \in \mathbb{Z}^n,\, y \in \mathbb{Z}^n_+,
\end{cases}
\end{align*}
where the last line was additionally multiplied by $\delta$ to become integer, and where $H_B^* = \delta H_B^{-1}$ is the adjoint matrix for $B$.

Finally, we transform the matrix $H_B$ into the SNF, such that $H_B = P^{-1} S Q^{-1}$, where $P^{-1}$, $Q^{-1}$ are unimodular matrices and $S$ is the SNF of $H_B$. After making the transformation $x \to Q x$, the initial problem becomes equivalent to the following problem:
\begin{align}
& w^\top x \to \min \label{GroupMin}\\
&\begin{cases}
G x \equiv g \,(\text{mod}\, S)\\
R x \leq r \\
x \in \mathbb{Z}_+^n,\, ||x||_{\infty} \leq n \Delta,
\end{cases}\notag
\end{align}
where $w^\top = - c^\top H_B^{-1}$, $G = P\text{ mod }S$, $g = P b_{1:n}\text{ mod }S$, $R = -H_N H_B^{*}$, and $r = \delta b_{(n+1) : m} - H_N H_B^{*} b_{1:n}$. The inequalities $||x||_{\infty} \leq n \Delta$ are additional tools to localize an optimal integer solution that follows from inequality \eqref{TardoshT}. Additionally, we have that $||R||_{\max} = ||H_N H_B^{*}||_{\max} \leq \Delta$.

Actually, the problem \eqref{GroupMin} is the classical Gomory's group minimization problem \cite{GOM65} (cf. \cite{HU70}) with an additional linear constraints. As in \cite{GOM65}, it can be solved using the dynamic programming approach. To this end, let us define the subproblems $Prob(l,\gamma,\eta)$:
\begin{align*}
& w_{1:l}^\top x \to \min\\
&\begin{cases}
G_{*\,1:l} x \equiv \gamma \,(\text{mod}\, S)\\
R_{*\,1:l} x \leq \eta \\
x \in \mathbb{Z}_+^l,
\end{cases}
\end{align*}
where $l \in 1:n$, $\gamma \in \Lambda(G)\text{ mod }S$, $\eta \in \mathbb{Z}^m$, and $||\eta||_{\infty} \leq n^2 \Delta^2$.

Let $\sigma(l,\gamma,\eta)$ be the objective function optimal value of $Prob(l,\gamma,\eta)$. When the problem $Prob(l,\gamma,\eta)$ is unfeasible, we put $\sigma(l,\gamma,\eta) = +\infty$. In the beginning, we put $\sigma(l,\gamma,\eta)=+\infty$, for all values $l$, $\gamma \not= 0$, $\eta \not= 0$. Trivially, the optimum of \eqref{IPP} is
\[
\sigma(n,g,\min\{r,\,n^2 \Delta^2 \vec 1\}).
\]

The following formula gives the relation between $\sigma(l,\cdot,\cdot)$ and $\sigma(l-1,\cdot,\cdot)$:
\[
\sigma(l,\gamma,\eta) = \min \{\sigma(l-1,\gamma - z G_{*\,l},\eta-z R_{*\,l}) +z w_l : |z| \leq n \Delta \}.
\]
The value of $\sigma(1,\gamma,\eta)$ can be computed using the following formula:
\[
\sigma(1,\gamma,\eta) = \min \{z w_1 : z G_{*\,1} \equiv \gamma \,(\text{mod } S),\, z R_{*\,1} \leq \eta,\, |z| \leq n \Delta \}.
\]

Both, the computational complexity of computing $\sigma(1,\gamma,\eta)$ and the reduction complexity of $\sigma(l,\gamma,\eta)$ to $\sigma(l-1,\cdot,\cdot)$, for all $\gamma$ and $\eta$, can be roughly estimated as:
\[
O( (\log \Delta + m) \cdot n \Delta^2 \cdot (n^2 \Delta^2)^m \cdot \mult(\log \Delta + \log n + \log ||w||_\infty) ).
\]
By Lemma \ref{SubRankDet}, $||w||_\infty \leq ||c||_1 \delta \log \delta$ and $\log ||w||_\infty = O(\log \Delta + \size(c))$. Finally, the result can be obtained multiplying the last formula by $n$.

\end{proof}

Let us consider the special case, when all $n \times n$ submatrices of $H$ are non-singular. In this case, by Lemma \ref{NRowsHNF}, for $n > \Delta(2 \Delta + 1)^2 + \log_2 \Delta$, the matrix $H$ can have at most $n+1$ rows ($m \leq 1$), and we have following corollary.

\begin{corollary}
If all $n \times n$ submatrices of $H$ are non-singular and $n > \Delta(2 \Delta + 1)^2 + \log_2 \Delta$, then the problem \eqref{IPP} can be solved by an algorithm with a complexity of
\[
O( \log \Delta \cdot n^4 \cdot \Delta^4 \cdot \mult(\size(c) + \log \Delta) ).
\]
\end{corollary}

\section{Simplex width computation}

Let $H \in \mathbb{Z}^{(n+1)\times n}$, $b \in \mathbb{Z}^{n+1}$, $\rank(H) = n$, and $P(H,b)$ be a simplex. Let us consider the problem of finding the width$(P(H,b))$ and a flat direction of $P(H,b)$.

The main result in \cite{GRIBC16} states that $\width(P(H,b))$ can be computed by an algorithm with a complexity of
\[
O(n^{2 \log \Delta_{n-1}(H)} \cdot \Delta(H) \cdot \Delta(H,b) \cdot \poly(n,\, \log \Delta(H,b)) ),
\] where $\Delta(H,b)$ is the maximum absolute value of $n \times n$ minors of the extended matrix $(H\,b)$.

In this section, we are going to develop an FPT-algorithm for the simplex width computation problem. Let us discuss our main tool.

Let $C \in \mathbb{Z}^{n \times n}$, $p \in \mathbb{Q}^n$, $\det(C) \not= 0$, $A \in \mathbb{Z}^{m \times n}$, $b \in \mathbb{Z}^n$, and $c \in \mathbb{Z}^n$. Suppose, for any $i \in 1:m$, one of the following equivalent conditions is true. 
\begin{align}
1)\,& \label{ConeCond1} ({A_{i\,*})}^\top \in \cone({(C^{-1})}^\top) \text{ and } c \in \cone(-{(C^{-1})}^\top),\\
2)\,& \label{ConeCond2} \quad p = \arg\min\{(A_{i\,*}) x : x \in p + \cone(C)\} = \\
&= \arg\max\{c^\top x : x \in p + \cone(C)\}, \notag \\
3)\,& \label{ConeCond3} \quad c^\top C \leq 0\text{ and }AC \geq 0_{m \times n}.
\end{align}

Let us consider the following problem that depends on the input vectors and the matrices $p,\,C,\,A,\,b,\,c$ with the conditions \eqref{ConeCond1}--\eqref{ConeCond3}.
\begin{align}
&c^\top x \to \max \label{ConeProg} \\
&\begin{cases}
x \in p + \cone(C) \\
x \in P(A,b) \cap \mathbb{Z}^n \\
\end{cases}\notag
\end{align}

The following lemma was proved in \cite{GRIBC16}, and it gives an algorithm for the problem \eqref{ConeProg}. 
\begin{lemma}\label{ConeProgLmOld}
There is an algorithm with a complexity of
\[
O(n^{2 \log \Delta(C)} \cdot \poly(n,\, \log \Delta(C),\, \size(A),\, \log ||b||_\infty,\, \log ||c||_\infty) )
\] to solve the problem \eqref{ConeProg}.
\end{lemma}

The main idea of the algorithm is the unimodular decomposition procedure from \cite{GRIBC16}. Actually, the technique based on the unimodular decomposition is very redundant, and it is better to use a simple procedure of enumerating integer points in some rational $n$-dimensional parallelepiped.

The following lemma (and the corresponding proof) is required to estimate the complexity of the enumeration procedure.

\begin{lemma}\label{ParEnumLm}
Let $A \in \mathbb{Q}^{n \times n}$, $p \in \mathbb{Q}^n$, $|\det(A)| = \Delta > 0$, and $M = p + \parl(A)$. Let, additionally, $A = Q H$, where $Q \in \mathbb{Z}^{n \times n}$ is an unimodular matrix and $H^\top$ is the HNF for $A^\top$ of the form \eqref{HNF}.

Then
\begin{equation}\label{ParNum}
\prod_{i = 1}^{n} \lfloor H_{i\,i} \rfloor \leq |M \cap \mathbb{Z}^n| \leq \prod_{i = 1}^{n} \lceil H_{i\,i} \rceil.
\end{equation}
\end{lemma}

\begin{proof} After the unimodular map $x \to Q^{-1} x$ the set $M$ becomes $M = r + \{ x \in \mathbb{R}^n : x = H t,\, t \in [0,1)^n \}$, where $r = Q p$. Let $y \in M \cap \mathbb{Z}^n$, then

\[
y_n = r_n + H_{n\,n} t_n, \qquad t_n = \frac{y_n - r_n}{H_{n\,n}},
\]

\[
t_n \in S_n =\{ \frac{\lceil r_n \rceil - r_n}{H_{n\,n}},\, \frac{\lceil r_n \rceil - r_n+1}{H_{n\,n}},\, \dots,\, \frac{ \lceil r_n \rceil - r_n + \lfloor H_{n\,n} \rfloor }{H_{n\,n}} \}.
\]

If $\lceil r_n \rceil - r_n \geq \{H_{n\,n}\}$, then the last element must be deleted from the set $S_n$, and $\lfloor H_{n\,n} \rfloor \leq |S_n| \leq \lceil H_{n\,n} \rceil$. Let $s = n - k$, for $k \in 1:n$. Then
\[
y_s = H_{s\,s} t_s + \tau_s, \qquad t_s = \frac{y_s - \tau_s}{H_{s\,s}},
\] where $\tau_s = r_s + \sum_{i=1}^{k} H_{s\,s+i} t_{s+ i}$. Finally, we have:
\[
t_s \in S_s = \{ \frac{\lceil \tau_s \rceil - \tau_s}{H_{s\,s}},\, \frac{\lceil \tau_s \rceil - \tau_s + 1}{H_{s\,s}},\,\dots,\, \frac{\lceil \tau_s \rceil - \tau_s + \lfloor H_{s\,s} \rfloor}{H_{s\,s}} \}.
\]

If $\lceil \tau_s \rceil - \tau_s \geq \{H_{s\,s}\}$, then the last element must be deleted from the set $S_s$, and $\lfloor H_{s\,s} \rfloor \leq |S_s| \leq \lceil H_{s\,s} \rceil$.
\end{proof}


\begin{lemma}\label{EnumComplx}
Let $A$ be the integral $n \times n$ matrix, $p \in \mathbb{Q}^n$, $|\det(A)| = \Delta > 0$. Then there is an algorithm with a complexity of $$O(\log \Delta \cdot n \Delta \cdot \mult(n \size(p) + \size(A) + n \log \Delta) + T_H(A))$$ to enumerate all integer points of the set $M = p + \parl(A)$, where $T_H(\cdot)$ is the HNF computational complexity.
\end{lemma}
\begin{proof} The proof of previous Lemma \ref{ParEnumLm} contains the enumeration algorithm, so we need only to estimate its complexity. Let $A = Q H$ and $r = Q p$ as in the proof of Lemma \ref{ParEnumLm}. Since $Q = A H^{-1}$, by Lemma \ref{SubRankDet}, we have $\size(r) = O(n \log \Delta + n \size(p) + \size(A))$.

Since $|y_i| \leq |H_{i\,*} t| \leq i |H_{i\,i}|$, we have $\size(y) = O(n \log n + \log \Delta)$ and $\size(y-r) = O(\size(r)+n \log n+\log \Delta) = O(\size(A) + n \size(p) + n \log \Delta)$.

Let $H^\prime$ be the matrix obtained from $H$ by replacing $j$-th column with column $y-r$. By Lemma \ref{SubRankDet}, we have $\size(\det H^\prime) = O(n \size(p) + \size(A) + n \log \Delta)$. Since $t_j = \frac{\det(H^\prime)}{\det(H)}$, we have $\size(t_j) = O(n \log \Delta + n \size(p) + \size(A))$, for any $j \in 1:n$.

Let $k$ be the number of diagonal elements of $H$ that are not equal to $1$, and $s = n - k$. Due to the proof of Lemma \ref{ParEnumLm}, we need
\[
O\Bigl(\sum_{i = 0}^{k} i \prod_{j=n-i}^n H_{j\,j}\Bigr) = O(\Delta k^2)
\] arithmetic operations to determine all possible values of the variables $y_i$ and $\tau_i$, for any $i \in (s+1):n$. When the values of $y_i$ have already been determined, for any $i \in (s+1):n$, then we can determine values of $\tau_i$ and $y_i = \lceil \tau_i \rceil$, for any $i \in 1:s$. The number of arithmetic operations for the last observation is $O(\Delta s k) = O(\Delta (n-k)k)$. Totally, we have
\[
O(\Delta k^2 + \Delta (n-k)k) = O(\log \Delta \cdot \Delta n)
\] arithmetic operations with values of a size of $O(n \size(p) + \size(A) + n \log \Delta)$. So, the total complexity becomes $O(\log \Delta \cdot n \Delta \cdot \mult(n \size(p) + \size(A) + n \log \Delta))$.
\end{proof}

Now, we can give a simple algorithm to determine the feasibility of the problem \eqref{ConeProg}.

\begin{lemma}\label{ConeProgLm}
There is an algorithm with a complexity of
\[
O(\Delta \cdot n^2 \cdot \mult(n \size(p) + \size(C) + \log ||A||_{\max} + n \log \Delta) + \Delta \size(b) + T_H(C))
\] to determine the feasibility of the problem \eqref{ConeProg}, where $\Delta = |\det(C)|$ and $m = O(n)$.
\end{lemma}
\begin{proof}
Let us show that the set $p + \parl(C)$ contain an optimal point of the problem \eqref{ConeProg}, if the set of feasible integer points is not empty. Let us consider the following decomposition:
\[
p + \cone(C) = \bigcup\limits_{z \in \mathbb{Z}^n_+} (p + C z + \parl(C)).
\]
For the purpose of deriving a contradiction, assume that the set $p + \parl(C)$ contains no optimal points. Let $x^*$ be an optimal point of the problem and $x^* \in p + C z + \parl(C)$, for $z \not= 0$. Then we have $y \in p + \parl(C)$, for the  point $y = x^* - C z$. By the condition \eqref{ConeCond3}, we have $c^\top C \leq 0$ and $A C \geq 0_{m \times n}$. Since $A C \geq 0_{m \times n}$ and $x^* \in P(A,b)$, we have  $y \in P(A,b)$. Since $c^\top C \leq 0$, we have $c^\top y \geq c^\top x^*$. The last two statements provide the contradiction.

Finally, we can use Lemma \ref{EnumComplx} to find an optimal point in the set $p + \parl(C)$. Each point $x \in p + \parl(C)$ must be checked by the condition $x \in P(A,b)$. The total complexity of the checking procedure is $$O(\Delta \cdot n m \cdot \mult(\log ||A||_{\max} + \log \Delta) + \Delta \size(b) ).$$
\end{proof}

It was shown in \cite{GRIBC16} (cf. Theorem 8 and Lemmas 4,5) that the width computation problem for the simplex $P(H,b)$ is equivalent to $O(n^2)$ feasibility problems of the following type:
\begin{equation}\label{NormProb}
(p^{(i)} + \cone(C)) \cap (q^{(i)} - \cone(C)) \cap \mathbb{Z}^{n-1},
\end{equation} where $p^{(i)},q^{(i)} \in \mathbb{Q}^{n-1}$, for $i \in 1:\gamma$, $C \in \mathbb{Z}^{(n-1)\times(n-1)}$ and
\begin{align}
&\gamma = O(n \Delta(H,b) \Delta(H)),\label{GammaTag}\\
&|\det(C)| \leq \Delta_{n-1}(H).\label{DetCTag}
\end{align}
The sizes, for $p^{(i)}$, $g^{(i)}$, and $C$, satisfy the following formulae:
\begin{align}
1)\,& \size(p^{(i)}) = O(n \log n + n \log \Delta(H,b)),\label{SizePQTag}\\
2)\,& \text{the same relation is true for }\size(q^{(i)}),\notag\\
3)\,& ||C||_{\max} \leq n\Delta^4(H,b),\label{SizeCTag}\\
4)\,& \size(C) = O(n^2 \log \Delta(H,b)).\notag
\end{align}

Now, we can prove the main result of the section.

\begin{theorem}
Let $H$ be an $(n+1) \times n$ integral matrix of the rank $n$ that have already been reduced to the HNF. Let $P(H,b)$ be a simplex, for $b \in \mathbb{Z}^{n+1}$, $\Delta = \Delta(H)$, and $\Delta(H,b)$ be the maximum absolute value of $n \times n$ minors of the augmented matrix $(H\,b)$.

The problem to compute $\width(P(H,b))$ and a flat direction of $P(H,b)$ can be solved by an algorithm with a complexity of
\[
O( \log \Delta \cdot n^5 \cdot \Delta^3 \cdot \Delta(H,b) \cdot \mult(n^3 \log \Delta(H,b) + n^3 \log n)).
\]
\end{theorem}
\begin{proof}
Let $C^* = \det(C) C^{-1}$ be the adjoint matrix of $C$. Since $$q^{(i)} - \cone(C) = P(C^*,C^*q^{(i)}),$$ the problem \eqref{NormProb} is equivalent to the problem
\begin{equation}\label{NormProbDual}
(p^{(i)} + cone(C)) \cap P(C^*,C^* q^{(i)}) \cap \mathbb{Z}^{n-1}.
\end{equation}
By Lemma \ref{SubRankDet} and the estimates \eqref{SizePQTag}, \eqref{SizeCTag}, we have $$||C^*||_{\max} \leq \Delta^2_{n-1}(H) \log \Delta_{n-1}(H) \leq 3 \Delta^4 \log^3 \Delta,$$ $\size(C^*) = O(n^2 \log \Delta)$ and $$\size(C^*q^{(i)}) = O(n \log\Delta + n \size(q^{(i)})) = O(n^2 \log n + n^2 \log \Delta(H,b)).$$

Hence, by Lemma \ref{ConeProgLm}, the feasibility problem \eqref{NormProbDual} can be solved by an algorithm with a complexity of
\[
O(T_H(C) + \log \Delta \cdot n^2 \cdot \Delta^2 \cdot \mult(n^3 \log \Delta(H,b) + n^3 \log n)).
\]

Let us note that the computational complexity for computing $C^*$ is $O(T_H(C))$, so we did not include it to the formula. There are $\gamma = O(n \Delta(H,b) \Delta)$ (cf. \eqref{GammaTag}) problems of that type, for any $i \in 1:\gamma$. And we are need to compute the HNF only one time, for each $C$. Therefore, the complexity becomes:
\[
O(T_H(C) + \log \Delta \cdot n^3 \cdot \Delta^3 \cdot \Delta(H,b) \cdot \mult(n^3 \log \Delta(H,b) + n^3 \log n)).
\]
Due to \cite{STORH96}, $T_H(C) = O^{\sim}(n^{\Theta} \mult(n \log ||C||_{\max}))$, where $\Theta$ is the matrix multiplication exponent and the symbol $O^{\sim}$ means that we omit some logarithmic factor. Hence, we can eliminate $T_H(C)$ from the complexity estimation. The final complexity result can be obtained multiplying the last formula by $n^2$, since the problem is equivalent to $O(n^2)$ subproblems of the type \eqref{NormProb}.
\end{proof}

Due to \cite{GRIBC16} (cf. Theorem 9), if additionally the simplex $P(H,b)$ is empty, or in other words $P(H,b) \cap \mathbb{Z}^n = \emptyset$, then $\gamma \leq \Delta$ (cf. \eqref{GammaTag}). This fact gives us a possibility to avoid an exponential dependence on $\size(b)$.

\begin{theorem}
If $P(H,b) \cap \mathbb{Z}^n = \emptyset$, then the problem to compute $\width(P(H,b))$ and a flat direction of $P(H,b)$ can be solved by an algorithm with a complexity of
\[
O( \log \Delta \cdot n^4 \cdot \Delta^4 \cdot \mult(n^3 \log \Delta(H,b) + n^3 \log n) ).
\]
\end{theorem}

\section*{Conclusion}
In Section 3, we presented FPT-algorithms for SLVP instances parameterized by the lattice determinant on lattices induced by near square matrices and on lattices induced by matrices without singular submatrices. Both algorithms can be applied to the $l_p$ norm, for any $p > 0$, and to the $l_\infty$ norm. In the future work, it could be interesting to develop FPT-algorithms for the SLVP for more general classes of norms defined by gauge functions $||\cdot||_K$, where $||x||_K = \inf\{s \geq 0: x \in s K\}$, $K$ is a convex body and $0 \in \inter(K)$.

In Section 4, we presented a FPT-algorithm for ILPP instances with near square constraints matrices parameterized by the maximum absolute value of rank minors of constraints matrices. Additionally, the last result gives us a FPT-algorithm for the case, when the ILPP constraints matrix has no singular rank submatrices, since these matrices can have only one additional row if the dimension is sufficiently large, due to \cite{AE16}. It is an interesting open problem to avoid the restriction for constraints matrices to be almost square and develop a FPT-algorithm for this case. It was mentioned in \cite{AW17} that the ILPP is NP-hard for values of parameter $\Delta = \Omega(n^\epsilon)$, for $\epsilon > 0$. So, the existence of a FPT-algorithm for the general class of matrices is unlikely.

In Section 5, we presented a FPT-algorithm for the simplex width computation problem parameterized by the maximum absolute value of rank minors of the augmented constraints matrix. The dependence on the augmented matrix minors can be avoided for empty lattice simplices. In the future work, it could be interesting to develop polynomial-algorithms or FPT-algorithms for wider types of polyhedra.

\section*{Acknowledgments}

\noindent

Results of Section 3 were obtained under financial support of Russian Science Foundation grant No 14-41-00039.

Results of Section 4 were obtained under financial support of Russian Science Foundation grant No 17-11-01336.

Results of Section 5 were obtained under financial support of Russian Foundation for Basic Research,
grant No 16-31-60008-mol-a-dk, and LATNA laboratory, NRU HSE.


\begin{thebibliography}{99}

\bibitem{AJTAI96} Ajtai,~M. (1996) Generating hard instances of lattice problems. Proceedings of 28th Annual ACM Symposium on the Theory of Computing
99--108.

\bibitem{AJKSK01} Ajtai,~M., Kumar,~R., Sivakumar,~D. (2001) A sieve algorithm for the shortest lattice vector problem. Proceedings of the 33rd Annual ACM
Symposium on Theory of Computing  601--610.

\bibitem{AJKSK02} Ajtai,~M., Kumar,~R., Sivakumar,~D. (2002) Sampling short lattice vectors and the closest lattice vector problem. Proceedings of 17th IEEE Annual
Conference on Computational Complexity 53--57.

\bibitem{AZ11} Alekseev,~V.~V., Zakharova,~D. (2011) Independent sets in the graphs with bounded minors of the extended incidence matrix. Journal of Applied and
Industrial Mathematics 5:14--18.

\bibitem{AE16} Artmann,~S., Eisenbrand,~F., Glanzer,~C., Timm,~O., Vempala,~S., Weismantel,~R. (2016) A note on non-degenerate integer programs with small
subdeterminants. Operations Research Letters 44(5):635--639.

\bibitem{AW17} Artmann,~S., Weismantel,~R., Zenklusen,~R. (2017) A strongly polynomial algorithm for bimodular integer linear programming. Proceedings of 49th Annual
ACM Symposium on Theory of Computing 1206--1219.

\bibitem{BLNAEW09} Bl\"omer,~J., Naewe,~S. (2009) Sampling methods for shortest vectors, closest vectors and successive minima. Theoretical Computer Science
410(18):1648--1665.

\bibitem{BOCK14} Bock,~A., Faenza,~Y., Moldenhauer,~C., Vargas,~R., Jacinto,~A. (2014) Solving the stable set problem in terms of the odd cycle packing number.
Proceedings of 34th Annual Conference on Foundations of Software Technology and Theoretical Computer Science 187--198.

\bibitem{BONY14} Bonifas,~N., Di~Summa,~M., Eisenbrand,~F., H\"ahnle,~N., Niemeier,~M. (2014) On subdeterminants and the diameter of polyhedra. Discrete \&
Computational Geometry 52(1):102--115.

\bibitem{CAS71} Cassels,~J.~W.~S. (1971) An introduction to the geometry of numbers, 2nd edition. Springer.

\bibitem{CHLEE15} Cheon,~J.~H., Lee,~C. (2015) Approximate algorithms on lattices with small determinant. Cryptology ePrint Archive, Report 2015/461,
http://eprint.iacr.org/2015/461.

\bibitem{COGST86} Cook,~W., Gerards,~A.~M.~H., Schrijver,~A., Tardos,~E. (1986) Sensitivity theorems in integer linear programming. Mathematical Programming 34:251--264.

\bibitem{PARAM15} Cygan,~M., Fomin,~F.~V., Kowalik,~L., Lokshtanov,~D., Marx,~D., Pilipczuk,~M., Pilipczuk,~M., Saurabh,~S. (2015) Parameterized algorithms. Springer.
\bibitem{DAD11} Dadush,~D., Peikert,~C., Vempala,~S. (2011) Enumerative algorithms for the shortest and closest lattice vector problems in any norm via M-ellipsoid coverings. 52nd IEEE Annual Symposium on Foundations of Computer Science 580--589.

\bibitem{EIS11} Eisenbrand,~F., H\"ahnle,~N., Niemeier,~M. (2011) Covering cubes and the closest vector problem. Proceedings of 27th Annual Symposium on Computational Geometry 417--423.

\bibitem{EIS16} Eisenbrand,~F., Vempala,~S. (2016) Geometric random edge. https://arxiv.org/abs/1404.1568v5.

\bibitem{PARAM99} Downey,~R.~G., Fellows,~M.~R. (1999) Parameterized complexity. Springer.

\bibitem{FP83} Fincke,~U., Pohst,~M. (1983) A procedure for determining algebraic integers of given norm. Lecture Notes in Computer Sceince 162:194--202.

\bibitem{FP85} Fincke,~U., Pohst,~M. (1985) Improved methods for calculating vectors of short length in a lattice, including a complexity analysis. Mathematics of Computation 44(170):463--471.

\bibitem{GOM65} Gomory,~R.~E. (1965) On the relation between integer and non-integer solutions to linear programs. Proceedings of the National Academy of Sciences of the United States of America 53(2):260--265.

\bibitem{GRIB13} Gribanov,~D.~V. (2013) The flatness theorem for some class of polytopes and searching an integer point. Springer Proceedings in Mathematics \& Statistics 104:37-45.

\bibitem{GRIBM17} Gribanov,~D.~V., Malyshev,~D.~S. (2017) The computational complexity of three graph problems for instances with bounded minors of constraint matrices. Discrete Applied Mathematics 227:13--20.

\bibitem{GRIBC16} Gribanov,~D.~V., Chirkov,~A.~J. (2016) The width and integer optimization on simplices with bounded minors of the constraint matrices. Optimization Letters 10(6):1179-1189.

\bibitem{GRIBV16} Gribanov,~D.~V., Veselov,~S.~I. (2016) On integer programming with bounded determinants. Optimization Letters 10(6):1169-1177.

\bibitem{GRUB87} Gruber,~M., Lekkerkerker,~C.~G. (1987) Geometry of numbers. North-Holland.

\bibitem{SVPSUR11} Hanrot,~G., Pujol,~X., Stehle,~D. (2011) Algorithms for the shortest and closest lattice vector problems. Lecture Notes in Computer Science 6639:159--190.

\bibitem{HU70} Hu,~T.~C. (1970) Integer programming and network flows. Addison-Wesley Publishing Company.


\bibitem{KANN83} Kannan,~R. (1983) Improved algorithms for integer programming and related lattice problems. Proceedings of 15th Annual ACM Symposium on Theory of Computing  99--108.

\bibitem{KANN87} Kannan,~R. (1987) Minkowski's convex body theorem and integer programming. Mathematics of Operations Research 12(3):415-440.

\bibitem{KAR84} Karmarkar,~N. (1984) A new polynomial time algorithm for linear programming. Combinatorica 4(4):373--391.

\bibitem{KHA80} Khachiyan,~L.~G. (1980) Polynomial algorithms in linear programming. Computational Mathematics and Mathematical Physics 20(1):53--72.

\bibitem{LLL82} Lenstra,~A.~K., Lenstra,~H.~W.~Jr., Lovasz,~L. (1982) Factoring polynomials with rational coefficients. Mathematische Annalen 261:515--534.

\bibitem{LEN83} Lenstra,~H.~W. (1983) Integer programming with a fixed number of variables. Mathematics of operations research 8(4):538--548

\bibitem{MICCVOUL10} Micciancio,~D., Voulgaris,~P. (2010) A deterministic single exponential time algorithm for most lattice problems based on Voronoi cell computations. Proceedings of 42nd Annual ACM Symposium on Theory of Computing  351--358.

\bibitem{NN94} Nesterov,~Y.~E., Nemirovsky,~A.~S. (1994) Interior point polynomial methods in convex programming. Society for Industrial and Applied Math, USA.

\bibitem{PAPA} Papadimitriou,~C.H. (1981) On the complexity of integer programming. Journal of the Association for Computing Machinery 28:765--768.

\bibitem{PAR91} Pardalos,~P.~M., Han,~C.~G., Ye,~Y. (1991) Interior point algorithms for solving nonlinear optimization problems. COAL Newsl. 19:45--54.

\bibitem{SEB99} Seb\"o,~A. (1999) An introduction to empty lattice simplicies. Lecture Notes in Computer Science 1610:400--414.

\bibitem{SIEG89} Siegel,~C.~L. (1989) Lectures on the geometry of numbers. Springer.

\bibitem{SHEV96} Shevchenko,~V.N. (1996) Qualitative topics in integer linear programming (translations of mathematical monographs). AMS Book.

\bibitem{SCHR98} Schrijver,~A. (1998) Theory of linear and integer programming. John Wiley \& Sons.

\bibitem{STORS96} Storjohann,~A. (1996) Near optimal algorithms for computing Smith normal forms of integer matrices. Proceedings of the 1996 International Symposium
on Symbolic and Algebraic Computation 267--274.

\bibitem{STORH96} Storjohann,~A., Labahn,~G. (1996) Asymptotically fast computation of Hermite normal forms of integer matrices. Proceedings of the 1996 International
Symposium on Symbolic and Algebraic Computation 259--266.

\bibitem{TAR86} Tardos,~E. (1986) A strongly polynomial algorithm to solve combinatorial linear programs. Operations Research 34(2): 250--256

\bibitem{VESCH09} Veselov,~S.~I., Chirkov,~A.~J. (2009) Integer program with bimodular matrix. Discrete Optimization 6(2): 220--222.

\bibitem{ZHEN05} Zhendong,~W. (2005) Computing the Smith forms of integer matrices and solving related problems. University of Delaware Newark, USA.

\end{thebibliography}
\end{document}